\documentclass[12pt]{amsart}
\usepackage{amsmath,amssymb,latexsym,cancel,rotating}
\usepackage{graphicx,amssymb,mathrsfs,amsmath,color,fancyhdr,amsthm}
\usepackage[all]{xy}
\usepackage{pgfplots}
\usepackage{tikz}
\usepackage{cite}
\usepackage{circuitikz-1.0}
\usepackage{verbatim}
\usepackage{float}
\usetikzlibrary{decorations.pathreplacing,decorations.markings}
\usetikzlibrary{arrows, decorations.pathmorphing,backgrounds,positioning,fit,petri}
\usetikzlibrary{decorations.pathreplacing,decorations.markings}
\usetikzlibrary{arrows.meta}

\newtheorem{theorem}{Theorem}[section]
\newtheorem{lemma}[theorem]{Lemma}
\newtheorem{corollary}[theorem]{Corollary}

\newtheorem{proposition}[theorem]{Proposition}

\newtheorem{remark}[theorem]{Remark}

\newtheorem{example}[theorem]{Example}

\begin{document}

\title[Maximal Rigid Representations of Continuous Quivers of Type $A$]
{Maximal Rigid Representations of Continuous Quivers of Type $A$}
\thanks{This work was supported by National Natural Science Foundation of China (No. 11771445)}

\author{Yanxiu Liu}             
\address{School of Science, Beijing Forestry University, Beijing 100083, P. R. China}
\email{liuyanxiu@bjfu.edu.cn (Y.Liu)} 

\author{Minghui ZHAO $^\dag$ }  
\address{School of Science, Beijing Forestry University, Beijing 100083, P. R. China}
\email{zhaomh@bjfu.edu.cn (M.Zhao)}
\thanks{$^\dag$ Corresponding author}

\subjclass[2010]{16G20}

\date{\today}

\keywords{maximal rigid representations, continuous quivers}

\bibliographystyle{abbrv}

\begin{abstract}
Bongartz and Gabriel gave a classification of maximal rigid representations for quivers of type $A$ with linear orientation and counted the number of isomorphism classes. In this paper, we give a formula on the number of isomorphism classes of a kind of maximal rigid representations for continuous quivers of type $A$ introduced by Igusa, Rock and Todorov. 
\end{abstract}

\maketitle

\section{Introduction}
Gabriel gave the notions of quivers and representations in \cite{gabriel1972unzerlegbare}. In \cite{gabriel1973indecomposable}, he gave a classification of indecomposable representations for finite type quivers. 
The first definition of tilting modules was introduced by Brenner and Butler in \cite{brenner2006generalizations}, and the definition used now was given by Happel and Ringel in \cite{happel1982tilted}.
In \cite{bongartz1982covering}, Bongartz and Gabriel gave a classification of tilting representations of quivers of type $A$ with linear orientation, and counted the number of isomorphism classes. Note that a representation is tilting and basic if and only if it is maximal rigid in this case. In \cite{buan2004tilting}, Buan and Krause gave a classification of tilting representations for cyclic quivers.

Representations of quivers of type $A$ play an important role in persistent homology,
which have been widely used in topological data analysis (see \cite{oudot2017persistence}). In \cite{carlsson2010zigzag}, Carlsson
and de Silva introduced zigzag persistent homology.
As a generalization of Gabriel's theorem for quivers of type $A$, Crawley-Boevey
in \cite{crawley2015decomposition} and Botnan in \cite{botnan2017interval} gave a classification of indecomposable representations of
$\mathbb{R}$ and infinite zigzag, respectively.
As a generalization, Igusa, Rock and Todorov introduced continuous quivers of type $A$ and  classified indecomposable representations in \cite{igusa2023continuous}.
In \cite{Sala_Schiffmann2019,Appel_Sala2020,Sala_Schiffmann2021,Appel_Sala_Schiffmann2022},
Appel, Sala and Schiffmann introduced continuum quivers independently. They also introduced the continuum Kac–Moody algebra and continuum quantum group associated to a continuum quiver.

In this paper, we study the number of isomorphism classes of maximal rigid representations of continuous quivers of type $A$. We define the type of representations and give a formula on the number of isomorphsim classes of maximal rigid representation of a given type for a continuous quiver. For a proof of this result, we bulid a connection between maximal rigid representations of continuous quivers and finite quivers.
  
In Section 2, we give notations and main results. Some results on tilting representations of finite quivers are recalled in Section 3. The proofs of main results are given in Section 4.
  
\section{Preliminaries}\label{section-1} 
\subsection{Continuous quivers}
Consider closed interval $[0,1]$ in $\mathbb{R}$ and the normal order $<$ on $[0,1]$. In this paper, $\mathbf{Q}=([0,1],<)$ is called a continuous quiver of type $A$.

\begin{remark}
The continuous quiver $\mathbf{Q}$  is a special case of continuous quivers $A_{\mathbb{R}}= (\mathbb{R}, \mathcal{S}, \preceq)$ of type $A$ in \cite{igusa2023continuous}. In this paper,  $[0,1]$, $\emptyset$ and $<$ are used instead of $\mathbb{R}$, $\mathcal{S}$,  and $\preceq$, respectively.

\end{remark}

Let $k$ be a fixed field. A representation of $\mathbf{Q}$ over $k$ is given by $\mathbf{M}=(\mathbf{M}(a), \mathbf{M}(a, b))$, where $\mathbf{M}(a)$ is a $k$-vector space for any $a \in [0, 1]$, and $\mathbf{M}(a, b):\mathbf{M}(a) \rightarrow \mathbf{M}(b)$ is a $k$-linear map for any $a < b  \in [0, 1]$.  Let $\mathbf{M}=(\mathbf{M}(a), \mathbf{M}(a, b))$ and $\mathbf{N}=(\mathbf{N}(a), \mathbf{N}(a, b))$ be representations of $\mathbf{Q}$. A family of $k$-linear maps $f=(f_a)_{a\in[0, 1]}$ is called a morphism from $\mathbf{M}$ to $\mathbf{N}$, if $f_b\mathbf{M}(a,b)=\mathbf{N}(a,b)f_a$ for any $a < b  \in [0, 1]$. Denote by $\text{Rep}_{k}(\mathbf{Q})$ the category of representations of $\mathbf{Q}$.

 For any $a < b  \in [0, 1]$, we use the notation $|a,b|$ for one of intervals $[a,b]$, $[a,b)$, $(a,b]$ and $(a,b)$.  Denote by $\mathbf{T}_{|a, b|}$ the following representation of $\mathbf{Q}$, where
 $$\mathbf{T}_{|a, b|}(x) = \left\{ 
    \begin{aligned}
    &k, & & x \in |a,b| \cr 
    &0, & & otherwise,
    \end{aligned}
\right.$$
and
$$\mathbf{T}_{|a, b|}(x,y) = \left\{ 
    \begin{aligned}
    &1_k, & &x<y \; and\; x,y \in |a,b| \cr 
    &0, & & otherwise.
    \end{aligned}
\right.$$

Denote by $\text{Rep}'_{k}(\mathbf{Q})$ the subcategory of $\text{Rep}_{k}(\mathbf{Q})$ consisting of representations with the following  decomposition $$\mathbf{M}=\bigoplus_ {i \in I}\mathbf{T}_{|x_i,y_i|}.$$


\subsection{Maximal rigid representations} 


Let $M$ be an object of category $\text{Rep}'_{k}(\mathbf{Q})$ with the following decomposition $$\mathbf{M}=\bigoplus_ {i \in I}\mathbf{M}_{i},$$
where $\mathbf{M}_{i}$ is indecomposable  for all $i\in I$.
The representation $\mathbf{M}$ is called basic if $\mathbf{M}_i \ncong \mathbf{M}_j$ for any $i \neq j \in I$, and 
is called rigid if $\text{Ext}^{1}(\mathbf{M}, \mathbf{M})=0$.  The representation $\mathbf{M}$ is called maximal rigid, if it is basic and rigid, and $\mathbf{M} \oplus \mathbf{N}$ is rigid implies that $\mathbf{N}$ $\in\text{add}\mathbf{M}$ for any object $\mathbf{N}$ of category $\text{Rep}'_{k}(\mathbf{Q})$.

Similarly to Lemma 3.1 in \cite{buan2004tilting}, we have the following lemma.

\begin{lemma}\label{lemma2-2}
Let $\mathbf{M}=\bigoplus\limits_ {i \in I}\mathbf{T}_{|x_i,y_i|}$ be an object of category $\text{Rep}'_{k}(\mathbf{Q})$. The representation $\mathbf{M}$ is rigid if and only if one of the following conditions is satisfied
\begin{enumerate}
    \item $|x_i,y_i|\subseteq|x_j,y_j|$;
    \item $|x_j,y_j|\subseteq|x_i,y_i|$;
    \item $y_i<x_j$;
    \item $y_j<x_i$;
    \item $y_i=x_j$, $|x_i,y_i|=|x_i,y_i)$ and $|x_j,y_j|=(x_j,y_j|$;
    \item $y_j=x_i$, $|x_j,y_j|=|x_j,y_j)$ and $|x_i,y_i|=(x_i,y_i|$,
\end{enumerate}
for any ${i\neq j \in I}$.
\end{lemma}

If two intervals $|a,b|$ and $|a',b'|$ satisfy one of the conditions in Lemma \ref{lemma2-2}, we say that they are compatible.

\subsection{Representations of type $\alpha$}
Let $\alpha= (a_0,a_1,\ldots,a_n)$ with $a_0=0$, $a_n=1$, and $a_i < a_{i+1}$ for any $0\leq i<n$.

For any $\mathbf{M}=\bigoplus\limits_ {i \in I}\mathbf{T}_{|x_i,y_i|}$ of the category $\text{Rep}'_{k}(\mathbf{Q})$, let $\hat{I}=\{|x_i,y_i||i\in I\}$.
For any $c \in (a_s, a_{s+1})$, denote
$$ \phi _{1}^{r}(c,\mathbf{M})= \{ d | a_{s+1} \leq d \leq 1, [c, d] \in \hat{I}\},$$$$\phi _{2}^{r}(c,\mathbf{M})= \{d |a_{s+1} \leq d \leq 1, [c, d) \in \hat{I}\},$$
$$ \phi _{3}^{r}(c,\mathbf{M})= \{d |a_{s+1} \leq d \leq 1 , (c, d] \in \hat{I}\},$$$$\phi _{4}^{r}(c,\mathbf{M})= \{d |a_{s+1} \leq d \leq 1, (c, d) \in \hat{I}\},$$
$$ \phi _{1}^{l}(c,\mathbf{M})= \{ d |0 \leq d \leq a_s, [d, c] \in \hat{I}\},$$$$\phi _{2}^{l}(c,\mathbf{M})= \{d |0 \leq d \leq a_s, (d, c] \in \hat{I}\},$$ 
$$ \phi _{3}^{l}(c,\mathbf{M})= \{d |0 \leq d \leq a_s,  [d, c)\in \hat{I}\},$$$$\phi _{4}^{l}(c,\mathbf{M})= \{d |0 \leq d \leq a_s, (d, c) \in \hat{I}\}.$$

The representation $\mathbf{M}$ is called of type $\alpha$, if it satisfies the following conditions:
\begin{enumerate} 
  \item for any $c\in [0,1]-\{a_0,\;  a_1 ,... ,\;  a_n\}$, we have 
  $$ \phi_{1}^{r}(c,\mathbf{M})  =\phi_{3}^{r}(c,\mathbf{M})\subset\{a_0,\;  a_1 ,... ,\;  a_n\} ,$$ $$ \; \phi_{2}^{r}(c,\mathbf{M})  =\phi_{4}^{r}(c,\mathbf{M}) \subset\{a_0,\;  a_1 ,... ,\;  a_n\},$$
  $$ \phi_{1}^{l}(c,\mathbf{M})  =\phi_{3}^{l}(c,\mathbf{M}) \subset\{a_0,\;  a_1 ,... ,\;  a_n\},\;$$ $$\phi_{2}^{l}(c,\mathbf{M}) =\phi_{4}^{l}(c,\mathbf{M})\subset\{a_0,\;  a_1 ,... ,\;  a_n\},$$
  $$|\phi_{1}^{r}(c,\mathbf{M})|  + |\phi_{2}^{r}(c,\mathbf{M})|  + |\phi_{1}^{l}(c,\mathbf{M})| + |\phi_{2}^{l}(c,\mathbf{M})| =1;$$ 
  \item for any $s=0,1,\ldots,n-1$ and $c, f \in (a_{s}, a_{s+1})$, we have
  $$\phi_{1}^{r}(c,\mathbf{M}) =\phi_{1}^{r}(f,\mathbf{M}), \;\phi_{2}^{r}(c,\mathbf{M}) =\phi_{2}^{r}(f,\mathbf{M}), $$$$\phi_{3}^{r}(c,\mathbf{M})=\phi_{3}^{r}(f,\mathbf{M}),\; \phi_{4}^{r}(c,\mathbf{M})=\phi_{4}^{r}(f,\mathbf{M}),$$ 
  $$\phi_{1}^{l}(c,\mathbf{M})=\phi_{1}^{l}(f,\mathbf{M}),\; \phi_{2}^{l} (c,\mathbf{M})=\phi_{2}^{l}(f,\mathbf{M}), $$$$ \phi_{3}^{l} (c,\mathbf{M})=\phi_{3}^{l}(f,\mathbf{M}), \;\phi_{4}^{l} (c,\mathbf{M})=\phi_{4}^{l}(f,\mathbf{M}).$$
   \end{enumerate}

\begin{remark}
By the definition, a representation is of type $\alpha$ implies that it is constant in intervals $(a_i,a_{i+1})$ for all $0\leq i<n$. Hence, we can build a connection between representations of type $\alpha$ and representations of finite quivers.
Note that this setting is similar to that in Section "Other Dynkin types" of Appendix B in \cite{Sala_Schiffmann2019}.
\end{remark}

\begin{example}\label{example_1}
Let $\alpha=(0,1)$.
Consider the following representations
  $$\mathbf{M}_1= \mathbf{T}_{[0,1]} \oplus \mathbf{T}_{(0,1]} \oplus \mathbf{T}_{[1,1]} \oplus\bigoplus_{x\in(0,1)} (\mathbf{T}_{[x,1]} \oplus \mathbf{T}_{(x,1]}),$$
    $$\mathbf{M}_2= \mathbf{T}_{[0,1]} \oplus \mathbf{T}_{(0,1]} \oplus \mathbf{T}_{[1,1]} \oplus \bigoplus_{x\in(0,1)} (\mathbf{T}_{(0,x]} \oplus \mathbf{T}_{(0,x)}),$$
     $$\mathbf{M}_3= \mathbf{T}_{[0,1]} \oplus \mathbf{T}_{(0,1]} \oplus \mathbf{T}_{(0,1)} \oplus \bigoplus_{x\in(0,1)} (\mathbf{T}_{[x,1)} \oplus \mathbf{T}_{(x,1)}),$$
       $$\mathbf{M}_4= \mathbf{T}_{[0,1]} \oplus \mathbf{T}_{(0,1]} \oplus \mathbf{T}_{(0,1)} \oplus \bigoplus_{x\in(0,1)} (\mathbf{T}_{(0,x)} \oplus \mathbf{T}_{(0,x]}),$$
         $$\mathbf{M}_5= \mathbf{T}_{[0,1]} \oplus \mathbf{T}_{[0,1)} \oplus \mathbf{T}_{(0,1)} \oplus \bigoplus_{x\in(0,1)} (\mathbf{T}_{[x,1)} \oplus \mathbf{T}_{(x,1)}),$$
           $$\mathbf{M}_6= \mathbf{T}_{[0,1]} \oplus \mathbf{T}_{[0,1)} \oplus \mathbf{T}_{(0,1)} \oplus \bigoplus_{x\in(0,1)} (\mathbf{T}_{(0,x]} \oplus \mathbf{T}_{(0,x)}),$$
             $$\mathbf{M}_7= \mathbf{T}_{[0,0]} \oplus \mathbf{T}_{[1,1]} \oplus \mathbf{T}_{[0,1]} \oplus\bigoplus_{x\in(0,1)} ( \mathbf{T}_{[x,1]} \oplus \mathbf{T}_{(x,1]}),$$  
               $$\mathbf{M}_8= \mathbf{T}_{[0,0]} \oplus \mathbf{T}_{[1,1]} \oplus \mathbf{T}_{[0,1]}  \oplus\bigoplus_{x\in(0,1)} ( \mathbf{T}_{[0,x]} \oplus \mathbf{T}_{[0,x)}),$$
                 $$\mathbf{M}_9= \mathbf{T}_{[0,0]} \oplus \mathbf{T}_{[0,1)} \oplus \mathbf{T}_{[0,1]} \oplus\bigoplus_{x\in(0,1)} ( \mathbf{T}_{[x,1)} \oplus \mathbf{T}_{(x,1)}),$$
                   $$\mathbf{M}_{10}=\mathbf{T}_{[0,0]} \oplus \mathbf{T}_{[0,1)} \oplus \mathbf{T}_{[0,1]}  \oplus\bigoplus_{x\in(0,1)} ( \mathbf{T}_{[0,x]} \oplus \mathbf{T}_{[0,x)}).$$
                   
The representations $\mathbf{M}_i$  are of type $\alpha$ for $i=1,2,\ldots,10$.
And the set $$\{[\mathbf{M}_i]|i=1,2,\ldots,10\}$$ gives all the isomorphism classes of maximal rigid representations of type $\alpha$.

\end{example}

The main result of this paper is the following theorem on the number of the isomorphism classes of maximal rigid representations of a given type.

\begin{theorem}\label{main-result}
Let $\mathbf{Q}$ be the continuous quiver of type $A$ and $\alpha= (a_0,a_1,\ldots,a_n)$ with $a_0=0$, $a_n=1$, and $a_i < a_{i+1}$ for any $0\leq i<n$.
Let $$\mathbf{A}=\{[\mathbf{M}] | \mathbf{M} \textrm{ is a maximal rigid representation of quiver } Q \textrm{  of type } \alpha\}.$$
Then$$|\mathbf{A}|= \frac{2^{n-1}}{n+1} \binom{4n+2}{2n+1}.$$
\end{theorem}

The proof of this theorem will be given in Section \ref{section-proof}.

In Example \ref{example_1}, the number of the isomorphism classes of maximal rigid representations is equal to $10$, which coincides with Theorem \ref{main-result}.

\section{Tilting representations for quivers of type A}

In this section, we shall recall some basic notations and results on tilting representations (see \cite{buan2004tilting,assem2006elements}).

Let $Q=(Q_0,Q_1,s,t)$ be the following quiver of type $A$, 
$$\bullet   \; \longrightarrow   \; \bullet   \; \longrightarrow   \;  \cdots   \; \longrightarrow   \; \bullet   \; \longrightarrow    \;\bullet   \;\quad $$
where $Q_0$ is the set of vertices, $Q_1$ is the set of arrows, and $s,t:Q_1\rightarrow Q_0$ are two maps such that $s(\rho)$ is the source of $\rho$ and $t(\rho)$ is the target of $\rho$.

Let $k$ be a fixed field. A representation of ${Q}$ over $k$ is given by ${M}=({M}(a), {M}(\rho))$, where ${M}(a)$ is a $k$-vector space for any $a \in Q_0$, and ${M}(\rho):{M}(s(\rho)) \rightarrow {M}(t(\rho))$ is a $k$-linear map for any $\rho\in Q_1$.  Let ${M}=({M}(a), {M}(\rho))$ and ${N}=({N}(a), {N}(\rho))$ be representations of $Q$. A family of $k$-linear maps $f=(f_a)_{a\in Q_0}$ is called a morphism from ${M}$ to ${N}$ if $f_{t(\rho)}{M}(\rho)={N}(\rho)f_{s(\rho)}$ for any $\rho\in Q_1$. Denote by $\text{Rep}_{k}({Q})$ the category of representations of ${Q}$.

A representation $M$ of $Q$ over $k$ is called basic if the direct summands of $M$ are not isomorphic to each other, and is called rigid if $\text{Ext}^{1}(M, M)=0$.  A representation $M$ is called maximal rigid, if it is basic and rigid, and $M \oplus N$ is rigid implies that $N$ $\in\text{add}M$ for any representation $N$ of $Q$.

A rigid representation $M$ is called tilting if the projective dimension of $M$ is equal to $1$, and there exists a short exact sequence $0 \to P \to M_i \to M_j \to 0$ with $M_i$, $M_j$ in add$M$ for any indecomposable projective representation $P$ of $Q$.

For any $a,b  \in Q_0$, denote $a\leq b$ if there is a path from $a$ to $b$ and $a<b$ if  $a\leq b$ but $a\neq b$. 
  For any $a\leq b  \in Q_0$, denote by $T_{[a, b]}$ the following representation of $Q$, where
 $$T_{[a, b]}(x) = \left\{ 
    \begin{aligned}
    &k, & & x \in Q_0 \textrm{ such that } a\leq x\leq b\cr 
    &0, & & otherwise,
    \end{aligned}
\right.$$
and
$${T}_{[a, b]}(\rho) = \left\{ 
    \begin{aligned}
    &1_k, & & a\leq s(\rho)< t(\rho)\leq b \cr 
    &0, & & otherwise.
    \end{aligned}
\right.$$


  
  

\begin{proposition} [\cite{bongartz1982covering,buan2004tilting} ]\label{prop_1}
Let $A=\{[M]|M\textrm{ is a basic and tilting representation of } Q\}.$
Then $$|A|= \frac{1}{n+1} \binom{2n}{n}.$$
\end{proposition}

The following lemma is well-known and we give a proof for convenience.
\begin{lemma}
Let $M$ be a representation of $Q$ over $k$. 
The representation $M$ is basic and tilting, if and only if $M$ is maximal rigid.
\end{lemma}
\begin{proof}
Assume that $M$ is basic and tilting, we shall prove that $M$ is maximal rigid. Obviously, $M$ is basic and rigid. Hence, we only need to prove that $M\oplus N$ is rigid implies  $N\in$add$M$.

Consider a projective resolution of $N$ $$0 \to P_0 \to P_1 \to N \to 0,$$ where $P_0$, $P_1$ are projective representations. Since $M$ is tilting,
there exists an exact sequence $$0 \to P_1 \to M_1 \to M_2 \to 0,$$ where  $M_1$, $M_2 \in$ add$M$. Since $M\oplus N$ is rigid, we get the following commutative diagram
$$\xymatrix{
                  &0 \ar[d]&0 \ar[d]\\
                  &P_0 \ar@{=}[r] \ar[d] &P_0 \ar[d]\\
                  0\ar[r]&P_1\ar[d] \ar[r] & M_1\ar[d] \ar[r]&M_2 \ar@{=}[d] \ar[r] &0  \\
                  0\ar[r]&N\ar[d] \ar[r] & N\oplus M_2\ar[d] \ar[r]&M_2 \ar[r] &0.  \\
                  &0 &0}$$

Similarly, since $M$ is tilting,
there exists an exact sequence $$0 \to P_0 \to M_3 \to M_4 \to 0,$$ where  $M_3$, $M_4 \in$ add$M$.
Since $M\oplus N$ is rigid, we have the following commutative diagram 
    $$\xymatrix{
                  &0 \ar[d]&0 \ar[d]\\
                  0\ar[r]&P_0\ar[d] \ar[r] & M_1\ar[d] \ar[r]&N\oplus M_2 \ar@{=}[d] \ar[r] &0  \\
                  0\ar[r]&M_3 \ar[d] \ar[r] & N \oplus M_2 \oplus M_3\ar[d] \ar[r]&N\oplus M_2 \ar[r] &0.  \\
                  &M_4 \ar@{=}[r] \ar[d] &M_4 \ar[d]\\
                  &0 &0
                  }$$
                  
Hence, we have $$M_1 \oplus M_4 \cong N \oplus M_2 \oplus M_3.$$ 
Since $M_1,M_2,M_3,M_4\in$ add$M$, we have $N\in$ add$M$. 

The converse is direct and we get the desired result.
\end{proof}

\begin{corollary}\label{cor_3.3}
Let $A=\{[M]|M\textrm{ is a maximal rigid representation of } Q\}$.
Then $$|A|= \frac{1}{n+1} \binom{2n}{n}.$$
\end{corollary}

\section{Proof of Theorem \ref{main-result}}\label{section-proof}
Let $\mathbf{Q}$ be the continuous quiver of type $A$, and $\alpha= (a_0,a_1,\ldots,a_n)$ with $a_0=0$, $a_n=1$, and $a_i < a_{i+1}$ for any $0\leq i<n$.
Denote $$\mathbf{A}=\{[\mathbf{M}]|\mathbf{M} \textrm{ is a maximal rigid representation of 
 } \mathbf{Q} \textrm{  of type } \alpha\}$$
and
$$\mathbf{B}=\{[\mathbf{M}]|\mathbf{M} \textrm{ is a rigid representation of } \mathbf{Q}\}.$$

Let $\widetilde{Q}$ be the following quiver.
$$\qquad  \bullet \;  \longrightarrow \; \bullet  \; \longrightarrow \;  \bullet  \; \longrightarrow   \; \bullet   \; \longrightarrow   \; \bullet  \; \longrightarrow  \; \cdots   \; \longrightarrow   \;  \bullet   \; \longrightarrow   \; \bullet   \; \longrightarrow    \;\bullet   \;\quad$$
$$  \qquad \qquad a_0 \; \qquad a_{0}^{+} \; \; \;  \; \quad a_{1}^{-} \;  \qquad  a_1 \qquad \; \; a_{1}^{+}  \quad \quad  \cdots  \;  \quad a_{n-1}^{+} \; \qquad a_{n}^{-}  \;\qquad a_{n}  \;\qquad \quad $$
Denote
$$\begin{aligned}
   \widetilde{B}=&\{[M]|M=\bigoplus_ {j\in I}{T}_{[x_j,y_j]} \textrm{ is a rigid representation of } \widetilde{Q}, \\
  & \textrm{and $x_j\neq a^-_i, y_j\neq a^+_i$ for any $j\in I$  and $i=0,1,\ldots,n$}\}.
\end{aligned}$$

Let $\widehat{Q}$ be the following quiver.
$$\qquad  \bullet \;  \longrightarrow \; \bullet  \; \longrightarrow \;  \bullet  \; \longrightarrow   \; \bullet   \; \longrightarrow\;\cdots \;\longrightarrow \; \bullet   \; \longrightarrow   \;  \bullet   \; \longrightarrow   \; \bullet   \; \longrightarrow    \;\bullet   \;\quad
$$
  $$  \qquad \qquad a_0 \; \qquad a_{01} \; \; \;  \; \quad a_1 \;  \qquad  a_{12} \quad \; \; ...  \quad \; a_{n-2,n-1}  \quad a_{n-1} \; \quad a_{n-1,n}  \;\quad a_{n}  \;\qquad \quad $$
Denote
$$\widehat{A} =\{[M] |  M\textrm{ is  a maximal rigid representation of quiver } \widehat{Q}\}$$
and
$$\widehat{B}=\{[M]|  M \textrm{ is a rigid representation of quiver } \widehat{Q}\}.$$

Consider a map $\tau_1:\mathbf{B}\to \widetilde{B}$ defined by
\begin{enumerate} 
  \item for any $x\in[0,1]-\{a_0,a_1,\ldots,a_n\}$, $$\tau_1([\mathbf{T}_{|a_i,x|}])=0 \textrm{ and } \tau_1([\mathbf{T}_{|x, a_i|}])=0;$$
  \item for any $0\leq i<j\leq n$,
  $$\tau_1([\mathbf{T}_{[a_i,a_j]}])=[{T}_{[a_i,a_j]}], $$$$\tau_1([\mathbf{T}_{(a_i,a_j]}])=[{T}_{[a_i^+,a_j]}],$$$$\tau_1([\mathbf{T}_{[a_i,a_j)}])=[{T}_{[a_i,a_j^-]}], 
  $$$$\tau_1([\mathbf{T}_{(a_i,a_j)}])=[{T}_{[a_i^+,a_j^-]}];$$
  \item for any $0\leq i\leq n$,
  $$\tau_1([\mathbf{T}_{[a_i,a_i]}])=[{T}_{[a_i,a_i]}];$$
  \item for any $[\mathbf{M}], [\mathbf{N}] \in\mathbf{B}$,
  $$\tau_1([\mathbf{M} \oplus \mathbf{N}])=[{M} \oplus {N}],$$
  where $[M]=\tau_1([\mathbf{M}])$ and $[N]=\tau_1([\mathbf{N}]).$
   \end{enumerate}

Consider a map $\tau_2:\widetilde{B} \to \widehat{B}$ defined by
\begin{enumerate} 
  \item for any $0\leq i<j\leq n$,
  $$\tau_2([{T}_{[a_i,a_j]}])=[{T}_{[a_i,a_j]}], $$$$\tau_2([{T}_{[a_i^+,a_j]}])=[{T}_{[a_{i,i+1},a_j]}],$$$$\tau_2([{T}_{[a_i,a_j^-]}])=[{T}_{[a_i,a_{j-1,j}]}], 
  $$$$\tau_2([{T}_{[a_i^+,a_j^-]}])=[{T}_{[a_{i,i+1},a_{j-1,j}]}];$$
  \item for any $0\leq i\leq n$,
  $$\tau_2([{T}_{[a_i,a_i]}])=[{T}_{[a_i,a_i]}];$$
  \item for any $[{M}], [{N}] \in\widetilde{B}$,
  $$\tau_2([{M}\oplus{N}])=[{M'} \oplus {N'}],$$
  where $[M']=\tau_2([{M}])$ and $[N']=\tau_2([{N}]).$
   \end{enumerate}

\begin{lemma}\label{lemma_4.1}
With above notations, we have
\begin{enumerate} 
  \item the map $\varphi=\tau_2\tau_1:\mathbf{B}\to\widehat{B}$ is surjective;
  \item $\varphi(\mathbf{A})=\widehat{A}$;
  \item for any $M\in\widehat{A}$, we have $|{\varphi}_{\mathbf{A}}^{-1}([M])|= 2^{n},$ where ${\varphi}_{\mathbf{A}}$ is the restriction of $\varphi$ on $\mathbf{A}$.
   \end{enumerate}
\end{lemma}

\begin{proof}
(1) Since $\tau_2$ gives a bijection between the sets of isomorphsim classes of indecomposable representations in $\widehat{B}$ and $\widetilde{B}$, it is a bijection between $\widehat{B}$ and $\widetilde{B}$.
It is direct that $\tau_1$ is surjective by the definition.
Hence, the map $\varphi$ is surjective. 

(2) Since $\mathbf{A}$ and $\widehat{A}$ are both defined by maximal conditions, this is direct.

(3)  Let $\mathbf{M}$ be a representation of $\mathbf{Q}$ such that $[\mathbf{M}]\in\mathbf{A}$ and $\varphi([\mathbf{M}])=[M]$. Then $$\mathbf{M}\cong\mathbf{T}\oplus\bigoplus_{j}\bigoplus_{x\in(a_j,a_{j+1})}
  \mathbf{T}_{x},$$
  where $$\mathbf{T}=\bigoplus_{s\leq t}\mathbf{T}^{\oplus t_{s,t}}_{|a_s,a_t|},$$
  and
  $$\mathbf{T}_{x}=\mathbf{T}_{x,1}=\bigoplus_{s\geq j+1}(\mathbf{T}_{[x,a_s]}\oplus\mathbf{T}_{(x,a_s]})^{\oplus t_{x,s}}\oplus\bigoplus_{s\geq j+1}(\mathbf{T}_{[x,a_s)}\oplus\mathbf{T}_{(x,a_s)})^{\oplus t'_{x,s}},$$
  or
  $$\mathbf{T}_{x}=\mathbf{T}_{x,2}=\bigoplus_{s\leq j}(\mathbf{T}_{[a_s,x]}\oplus\mathbf{T}_{[a_s,x)})^{\oplus t_{x,s}}\oplus\bigoplus_{s\leq j}(\mathbf{T}_{(a_s,x]}\oplus\mathbf{T}_{(a_s,x)})^{\oplus t'_{x,s}},$$
 such that $$\sum_{0\leq s\leq n}(t_{x,s}+t'_{x,s})=1.$$

  Note that $\mathbf{T}$ is determinated by $[M]$ uniquely, $\varphi([\mathbf{T}])=[M]$ and $\varphi([\mathbf{T}_{x}])=0$.

  Let $X=\{|a_s,a_t||t_{s,t}>0\}$. Fix $x\in(a_j,a_{j+1})$ for some $j$. 
  For any $j+1\leq l\leq m$, if $[x,a_l|$ and $[x,a_m|$ are two different intervals compatible with all the interval in $X$, then $[a_{j+1},a_m|$ is compatible with all the interval in $X$. Hence, the maximal condition implies that
  $$\mathbf{T}_{x,1}=\mathbf{T}_{[x,a_s|}\oplus\mathbf{T}_{(x,a_s|}$$
  for some $s$. Similarly, $$\mathbf{T}_{x,2}=\mathbf{T}_{|a_{s'},x]}\oplus\mathbf{T}_{|a_{s'},x)}$$
  for some $s'$.
  Hence $\mathbf{T}_{x}$ has two choices.
  Note that $t_{x,s}=t_{y,s}$ and $t'_{x,s}=t'_{y,s}$ for $x,y\in(a_j,a_{j+1})$.
  Hence $|{\varphi}_{\mathbf{A}}^{-1}([M])|= 2^{n}.$
\end{proof}






\begin{example}
Let $\alpha=(0,1)$.
In Example \ref{example_1}, we have $$\mathbf{A}=\{[\mathbf{M}_i]|i=1,2,\ldots,10\}.$$

Now, $\widehat{Q}$ is the following quiver
$$\qquad  \bullet \;  \longrightarrow \; \bullet  \; \longrightarrow \;  \bullet$$
$$   \qquad a_0 \; \qquad a_{01} \; \; \;  \; \quad a_1$$ 
and  
$$\widehat{A}=\{[{M}_i]|i=1,2,\ldots,5\},$$
where
    $${M}_1= T_{[a_{0},a_{1}]} \oplus T_{[a_{1},a_{1}]} \oplus T_{[a_{01},a_{1}]},$$
   $${M}_2= T_{[a_{0},a_{1}]} \oplus T_{[a_{01},a_{01}]} \oplus T_{[a_{01},a_{1}]},$$ 
   $${M}_3= T_{a_{0},a_{1}]} \oplus T_{[a_{01},a_{01}]} \oplus T_{[a_{0},a_{01}]},$$ 
   $${M}_4= T_{[a_{0},a_{0}]} \oplus T_{[a_{1},a_{1}]} \oplus T_{[a_{0},a_{1}]},$$ 
   $${M}_5= T_{[a_{0},a_{0}]} \oplus T_{[a_{0},a_{01}]} \oplus T_{[a_{0},a_{1}]}.$$

The map $\varphi: \mathbf{A} \to  \widehat{A}$ satisfies $\varphi(\mathbf{M}_{2k})= \varphi(\mathbf{M}_{2k-1})=M_k$
for $k=1,2,...,5$. Hence, $|{\varphi}_{\mathbf{A}}^{-1}(M)|=2$, for any $M\in\widehat{A}$,
which coincides with Lemma \ref{lemma_4.1}.    
\end{example}

Then we shall give a proof of Theorem \ref{main-result}.
\begin{proof}
Corollary \ref{cor_3.3}  implies that
$$|\widehat{A}|=\frac{1}{2n+2} \binom{4n+2}{2n+1}.$$
Hence, we have $$|\mathbf{A}|=2^{n}|\widehat{A}|=\frac{2^{n-1}}{n+1}\binom{4n+2}{2n+1}$$
by Lemma \ref{lemma_4.1}.
\end{proof}

\bibliography{bibfile}

\begin{thebibliography}{10}

\bibitem{Appel_Sala2020}
A.~Appel and F.~Sala.
\newblock Quantization of continuum {K}ac–{M}oody algebras.
\newblock {\em Pure and Applied Mathematics Quarterly}, 16(3):439--493, 2020.

\bibitem{Appel_Sala_Schiffmann2022}
A.~Appel, F.~Sala, and O.~Schiffmann.
\newblock Continuum {K}ac–{M}oody algebras.
\newblock {\em Moscow Mathematical Journal}, 22(2):177--224, 2022.

\bibitem{assem2006elements}
I.~Assem, D.~Simson, and A.~Skowronski.
\newblock {\em Elements of the representation theory of associative algebras}.
\newblock Cambridge University Press, 2006.

\bibitem{bongartz1982covering}
K.~Bongartz and P.~Gabriel.
\newblock Covering spaces in representation-theory.
\newblock {\em Inventiones Mathematicae}, 65:331--378, 1982.

\bibitem{botnan2017interval}
M.~Botnan.
\newblock Interval decomposition of infinite zigzag persistence modules.
\newblock {\em Proceedings of the American Mathematical Society}, 145(8):3571--3577, 2017.

\bibitem{brenner2006generalizations}
S.~Brenner and M.~C. Butler.
\newblock Generalizations of the {B}ernstein-{G}elfand-{P}onomarev reflection functors.
\newblock In {\em Representation Theory II: Proceedings of the Second International Conference on Representations of Algebras Ottawa, Carleton University, August 13--25, 1979}, pages 103--169. Springer, 2006.

\bibitem{buan2004tilting}
A.~B. Buan and H.~Krause.
\newblock Tilting and cotilting for quivers of type ${A}_n$.
\newblock {\em Journal of Pure and Applied Algebra}, 190(1-3):1--21, 2004.

\bibitem{carlsson2010zigzag}
G.~Carlsson and V.~De~Silva.
\newblock Zigzag persistence.
\newblock {\em Foundations of Computational Mathematics}, 10:367--405, 2010.

\bibitem{crawley2015decomposition}
W.~Crawley-Boevey.
\newblock Decomposition of pointwise finite-dimensional persistence modules.
\newblock {\em Journal of Algebra and its Applications}, 14(05):1550066, 2015.

\bibitem{gabriel1972unzerlegbare}
P.~Gabriel.
\newblock Unzerlegbare darstellungen {I}.
\newblock {\em Manuscripta Mathematica}, 6:71--103, 1972.

\bibitem{gabriel1973indecomposable}
P.~Gabriel.
\newblock Indecomposable representations {II}.
\newblock {\em Sympos. Math., Convegno di Algebra Commutativa, INDAM, Roma, 1971}, pages 81--104, 1973.

\bibitem{happel1982tilted}
D.~Happel and C.~M. Ringel.
\newblock Tilted algebras.
\newblock {\em Transactions of the American Mathematical Society}, 274(2):399--443, 1982.

\bibitem{igusa2023continuous}
K.~Igusa, J.~D. Rock, and G.~Todorov.
\newblock Continuous quivers of type {A} ({I}) foundations.
\newblock {\em Rendiconti del Circolo Matematico di Palermo Series 2}, 72(2):833--868, 2023.

\bibitem{oudot2017persistence}
S.~Y. Oudot.
\newblock {\em Persistence theory: from quiver representations to data analysis}.
\newblock American Mathematical Society, 2017.

\bibitem{Sala_Schiffmann2019}
F.~Sala and O.~Schiffmann.
\newblock The circle quantum group and the infinite root stack of a curve.
\newblock {\em Selecta Mathematica}, 25(77):1--86, 2019.

\bibitem{Sala_Schiffmann2021}
F.~Sala and O.~Schiffmann.
\newblock Fock space representation of the circle quantum group.
\newblock {\em International Mathematics Research Notices}, 2021(22):17025--17070, 2021.

\end{thebibliography}

\end{document}